\pdfoutput=1
\RequirePackage{ifpdf}
\ifpdf 
\documentclass[pdftex]{sigma}
\else
\documentclass{sigma}
\fi

\usepackage{mathrsfs}
\usepackage{tensor}
\usepackage{youngtab}

\theoremstyle{plain}
\newtheorem{Observation}{Observation}
\newenvironment{ObservationBis}[1]{
	\renewcommand{\theObservation}{\ref*{#1}$'$}%
	\addtocounter{Observation}{-1}%
	\begin{Observation}
}{
	\end{Observation}
}
\newtheorem{fact}{Fact}
\newtheorem*{maintheorem}{Main Theorem}
\newtheorem*{mainproblem}{Main Problem}
\newtheorem{problem}{Problem}

\numberwithin{equation}{section}

\renewcommand{\P}{\mathbb P}
\renewcommand{\S}{\mathbb S}

\DeclareMathOperator{\OG}{O}
\DeclareMathOperator{\SOG}{SO}

\newcommand{\isom}{\cong}
\newcommand{\Z}{\mathbb Z}
\newcommand{\R}{\mathbb R}

\begin{document}

\allowdisplaybreaks

\renewcommand{\thefootnote}{$\star$}

\newcommand{\arXivNumber}{1510.09028}

\renewcommand{\PaperNumber}{041}

\FirstPageHeading

\ShortArticleName{Are Orthogonal Separable Coordinates Really Classif\/ied?}

\ArticleName{Are Orthogonal Separable Coordinates\\ Really Classif\/ied?\footnote{This paper is a~contribution to the Special Issue
on Analytical Mechanics and Dif\/ferential Geometry in honour of Sergio Benenti.
The full collection is available at \href{http://www.emis.de/journals/SIGMA/Benenti.html}{http://www.emis.de/journals/SIGMA/Benenti.html}}}

\Author{Konrad SCH\"OBEL}
\AuthorNameForHeading{K.~Sch\"obel}
\Address{Mathematisches Institut, Fakult\"at f\"ur Mathematik und Informatik,\\
Friedrich-Schiller-Universit\"at Jena, 07737 Jena, Germany}
\Email{\href{mailto:konrad.schoebel@uni-jena.de}{konrad.schoebel@uni-jena.de}}

\ArticleDates{Received October 30, 2015, in f\/inal form March 15, 2016; Published online April 26, 2016}

\Abstract{We prove that the set of orthogonal separable coordinates on an arbitrary
(pseudo-)Riemannian manifold carries a natural structure of a
projective variety, equipped with an action of the isometry group. This
leads us to propose a new, algebraic geometric approach to the
classif\/ication of orthogonal separable coordinates by studying the
structure of this variety. We give an example where this approach reveals
unexpected structure in the well known classif\/ication and pose a number of
problems arising naturally in this context.}


\Classification{14H70; 53A60; 58D27}

\renewcommand{\thefootnote}{\arabic{footnote}}
\setcounter{footnote}{0}

\section{Introduction}

``Of course, they are!'' any reader familiar with the topic will immediately
reply to the question in the title, ``The problem has been solved exhaustively
almost 30 years ago''. Indeed, based on the pioneering works of Paul
St\"ackel~\cite{Staeckel} and Luther P.~Eisenhart~\cite{Eisenhart}, Ernest
G.~Kalnins and Willard Miller gave a comprehensive solution of the classical
problem to classify all orthogonal separable coordinate systems on a constant
curvature manifold \cite{Kalnins, Kalnins&Miller}. These are coordinates in
which fundamental equations like the Hamilton--Jacobi equation or the
Schr\"odinger equation can be solved by a separation of variables. With our
somewhat provocative title we do by no means at all want to cast any doubt on
this milestone in the theory of separation of variables, but rather question
what is meant by the word \emph{solved}. Our reader will again reply that
Kalnins and Miller's result gives a complete list of (families of equivalence
classes of) orthogonal separable coordinates. The issue we would like to
address here is that this list, mathematically speaking, is nothing but a mere
set, whereas we will show that the combination of some very elementary
observations leads to the conclusion that this set actually has a~rich
geometric structure.
Namely, we prove:

\begin{maintheorem}
	The set of orthogonal separable coordinates on an $n$-di\-men\-sion\-al
	\mbox{$($pseu\-do-$)$}Rie\-mannian manifold $M$ carries a natural structure of a
	projective variety, isomorphic to a subvariety $\mathcal S(M)\subseteq
	G_n\bigl(\mathcal K(M)\bigr)$ in the Grassmannian of $n$-planes in the
	space $\mathcal K(M)$ of second order Killing tensors on~$M$. Moreover,
	it comes equipped with a natural action of the isometry group.
\end{maintheorem}

We call $\mathcal S(M)$ the \emph{variety of orthogonal separable coordinates}
on~$M$.

Our theorem shows that the classif\/ication problem for orthogonal separable
coordinates is, essentially, an algebraic geometric problem and should
therefore be treated by algebraic geometric rather than dif\/ferential geometric
methods. To the best of our knowledge, this point has never been made in the
vast amount of literature on separation of variables. Separable coordinates
have always been studied via partial dif\/ferential equations, while our result
implies that it is governed by algebraic equations alone. This observation
could have been made over 60 years ago, since it is mainly a consequence of the
Nijenhuis integrability conditions, published in 1951~\cite{Nijenhuis}.

So, from a categorical point of view, we can say that the classif\/ication
problem for orthogonal separable coordinates \emph{is solved}, but in the
category of sets and \emph{not in its natural category}, which is the category
of projective varieties equipped with group actions. This motivates us to propose an algebraic geometric approach to the
classif\/ication of orthogonal separable coordinates, by investigating the
variety ${\mathcal S}(M)$ as in~\cite{Schoebel15} for constant curvature manifolds.

\begin{mainproblem}
	For a given $($pseudo-$)$Riemannian manifold~$M$, describe the structure of
	the variety $\mathcal S(M)$ of orthogonal separable coordinate systems and
	of its quotient $\mathcal S(M)/G$ under the isometry group~$G$.
\end{mainproblem}

What does it mean to describe the structure of this variety and why should
that give any relevant information for the classif\/ication of separation
coordinates? We will illustrate this using the example of the $n$-dimensional
sphere $\S^n$, for which the variety of orthogonal separable coordinates in
normal form has recently been identif\/ied to the Deligne--Mumford moduli space
$\bar{\mathscr M}_{0,n+2}(\R)$ of stable algebraic curves of genus zero with
$n+2$ marked points~\cite{Schoebel&Veselov}. Exploiting the properties of
these spaces reveals a rich and unexpected geometric, combinatorial and
algebraic structure behind the long known classif\/ication.

While some may have considered separation of variables as ``old-fashioned'' or
even ``outdated'', this example builds a bridge from the over 120 years old
theory of separation of variables to modern algebraic geometry and popular
topics like operad theory. Spheres are but the simplest example of constant
curvature manifolds in this context and even more interesting results are
expected for other manifolds and other notions of variable separation. That
is to say, we are far from a complete understanding of variable separation in
all its facets.

To keep this notice concise, we will refrain from giving a more elaborate
introduction to separation of variables, which can be found elsewhere in the
abundant literature on this subject. We concentrate instead on those few
def\/initions and results necessary to set up our algebraic geometric
description.

\section{Some simple observations}

\subsection{Killing tensors}

{\sloppy Our f\/irst two observations concern second order Killing tensors, the
fundamental object in the characterisation of separable coordinates. These
are symmetric tensor f\/ields $K_{\alpha\beta}$ on a~\mbox{(pseudo-)}Rie\-mannian
manifold which satisfy the \emph{Killing equation}
\begin{gather}
	\label{eq:Killing}
	\nabla_aK_{bc}+
	\nabla_bK_{ca}+
	\nabla_cK_{ab}=0.
\end{gather}
Here $\nabla$ denotes the covariant derivative with respect to the Levi-Civita
connection of the metric. Note that the metric itself is trivially a Killing
tensor, since it is covariantly constant.

}

\begin{Observation}
	The Killing equation~\eqref{eq:Killing}	is linear.
	Therefore Killing tensors form a vector space.
\end{Observation}

\begin{Observation}
	The Killing equation \eqref{eq:Killing} is overdetermined of finite type.
	Hence the space of Killing tensors is finite-dimensional.
\end{Observation}

\begin{definition}
	We denote the f\/inite-dimensional space of second order Killing tensors on
	a~(pseu\-do-)Riemannian manifold~$M$ by~$\mathcal K(M)$.
\end{definition}

\subsection{St\"ackel systems}

Here and in what follows we will silently identify symmetric forms with
endomorphisms using the metric. In particular, we can interpret a Killing
tensor as a f\/ield of symmetric endomorphisms. We say that two Killing tensors
commute if the corresponding endomorphisms commute algebraically at every
point. Note that the metric corresponds to the identity endomorphism and
hence commutes with any other Killing tensor.

A f\/ield of symmetric endomorphisms is said to have \emph{surface normal
eigenvectors}, if it has simple eigenvalues and the orthogonal complement of
each eigenvector f\/ield forms an integrable distribution\footnote{Some care has to be taken on a pseudo-Riemannian manifold, where the
	eigenvalues and eigenvectors may be complex. In this case either the
	separable coordinates are not everywhere def\/ined or one weakens the
	notion of variables separation to allow for complex variables.}.
By def\/inition such a f\/ield def\/ines a family of $n$ orthogonal hypersurface
foliations. Moreover, any symmetric endomorphism f\/ield with surface normal
eigenvectors and commuting with the f\/irst def\/ines the same foliations.

\begin{definition}
	A \emph{St\"ackel system} on an $n$-dimensional (pseudo-)Riemannian
	manifold is an $n$-dimensional vector space spanned by the metric and
	$n-1$ mutually commuting Killing tensors with surface normal
	eigenvectors.
\end{definition}

The leaves of a family of $n$ complementary hypersurface foliations are the
coordinate hypersurfaces of a coordinate system which is unique up to
permutations and reparametrisations of the individual coordinates. By abuse
of language, we will call an equivalence class of coordinate systems under
these transformations simply \emph{coordinates}. Moreover, if the
Hamilton--Jacobi equation
\begin{gather*}
	\tfrac12g^{ij}\frac{\partial W}{\partial x^i}\frac{\partial W}{\partial x^j}=E
\end{gather*}
admits a solution of the form
\begin{gather*}
	W\big(x^1,\ldots,x^n;\underline c\big)=W_1\big(x^1;\underline c\big)+\cdots+W_n\big(x^n;\underline c\big),
	\qquad
	\det\left(\frac{\partial^2W}{\partial x^i\partial c_j}\right)\not=0
\end{gather*}
in the coordinates $x^1,\ldots,x^n$, we call them \emph{separable
coordinates}. We can now state the classical result, which is the basis for
their classif\/ication.

\begin{theorem}[\cite{Eisenhart, Staeckel}]
	Locally, there is a bijective correspondence between St\"ackel systems and
	orthogonal separable coordinates for the Hamilton--Jacobi equation.
\end{theorem}

Although the above correspondence is only local, we can use it to classify
orthogonal separable coordinates which are global in the sense that they are
induced by a global St\"ackel system. In particular, it gives a global
classif\/ication if every local Killing tensor extends to a global one, as is
the case with simply connected constant curvature manifolds.

We should remark that an (additive) separation of the Hamilton--Jacobi equation
is a necessary condition for a (multiplicative) separation of the Schr\"odinger
and other prominent equations and that suf\/f\/icient conditions can be given
\cite{Eisenhart, Robertson}. This is why the classical result is formulated
for the Hamilton--Jacobi equation. We conclude with an algebraic geometric
view on St\"ackel systems.

\begin{Observation}
	By definition, a~St\"ackel system defines a point in the Grassmannian
	$G_n(\mathcal K(M))$ of $n$-planes in the space of Killing
	tensors.
\end{Observation}

\subsection{Nijenhuis integrability}

The above geometric def\/inition of surface normal eigenvectors is not suitable
for our purposes. We will instead use an analytic characterisation due to
Nijenhuis \cite{Nijenhuis}, who showed that a diagonalisable endomorphism
f\/ield $K$ has surface normal eigenvectors if and only if it satisf\/ies the
following \emph{Nijenhuis integrability conditions}:
\begin{subequations}
	\label{eq:Nijenhuis}
	\begin{gather}
		K\indices{_{d[a}}\nabla_bK\indices{_{c]}^d} =0,\\
		K\indices{_d^e}K\indices{_{e[a}}\nabla_bK\indices{_{c]}^d}+K\indices{_{d[a}}K\indices{^e_b}\nabla^dK\indices{_{c]e}} =0,\\
K\indices{_d^e}K\indices{_e^f}K\indices{_{f[a}}\nabla_bK\indices{_{c]}^d}+K\indices{_e^f}K\indices{_{d[a}}K\indices{^e_b}\nabla^dK\indices{_{c]f}} =0.
	\end{gather}
\end{subequations}

\begin{definition}
	We call a Killing tensor satisfying the Nijenhuis integrability conditions
	\eqref{eq:Nijenhuis} \emph{Nijenhuis integrable} and denote the set of
	Nijenhuis integrable Killing tensors on a~\mbox{(pseudo-)}Rie\-mannian manifold
	$M$ by $\mathcal I(M)$.
\end{definition}

The next observation is due to Sergio Benenti~\cite{Benenti}.

\begin{Observation}
	\label{obs:Benenti}
	Every Killing tensor with surface normal eigenvectors is contained in a~unique St\"ackel system.
\end{Observation}

Note that this observation would reduce the classif\/ication of separable
coordinates to solving the Nijenhuis conditions for a (single) Killing tensor,
since the determination of the corresponding St\"ackel system is only a linear
problem. Nevertheless, a direct solution of this system of non-linear partial
dif\/ferential equations has been considered intractable~\cite{Horwood&McLenaghan&Smirnov}. In contrast, our next two observations
show that these equations have a remarkable property.

\begin{Observation}
	The covariant derivative $K\mapsto\nabla K$ is a linear operation. Hence
	any polynomial expression in~$K$ and~$\nabla K$ is actually polynomial in
	$K$ alone.
\end{Observation}

\begin{Observation}	\label{obs:last}
	The Nijenhuis equations are homogeneous algebraic equations in $($the
	components of$)$ $K$ and its covariant derivative~$\nabla K$. This implies
	that the Nijenhuis integrability conditions define polynomial equations on
	the space~$\mathcal K(M)$ of Killing tensors.
\end{Observation}

Let us reformulate this in algebraic geometric terms.

\begin{ObservationBis}{obs:last}
	The set $\mathcal I(M)$ of Nijenhuis integrable Killing tensors forms a
	projective variety in the space $\mathcal K(M)$ of Killing tensors.
\end{ObservationBis}

A priori, the Nijenhuis conditions yield inf\/initely many algebraic equations
on $\mathcal K(M)$, as they have to be satisf\/ied at every point of $M$. By
Hilbert's basis theorem these can be reduced to a~f\/inite set.
For an explicit example see~\cite{Schoebel12}.
Note that we
will not, as usually done, exclude Nijenhuis integrable Killing tensors with
multiple eigenvalues since this would destroy the property of $\mathcal I(M)$
to be a variety.

By def\/inition, a St\"ackel system is a projective $(n-1)$-plane contained in the
variety $\mathcal I(M)$. Observation~\ref{obs:Benenti} then tells us that
$\mathcal I(M)$ has a very particular structure: It is ruled by
$(n-1)$-planes. For an explicit example see~\cite{Schoebel14}.

\subsection{Isometry group action}

We seek to classify separable coordinates up to isometries. We will therefore
conclude this section with three observations concerning the isometry group
action.

\begin{Observation}
	The Killing equation~\eqref{eq:Killing} is invariant under isometries, so
	that the isometry group acts linearly on the space of Killing tensors.
\end{Observation}

\begin{Observation}
	The Nijenhuis integrability conditions \eqref{eq:Nijenhuis} and therefore
	the set $\mathcal I(M)$ of Nijenhuis integrable Killing tensors are
	invariant under isometries.
\end{Observation}

\begin{Observation}
	The property that two Killing tensors commute is invariant under
	isometries.
\end{Observation}

Combining the preceding two observations, we arrive at:

\begin{Observation}\label{obs:groupaction}
	The isometry group maps St\"ackel systems to St\"ackel systems.
\end{Observation}

\section{The foundation}

We now put together all observations made so far in order to prove our Main
Theorem. We will employ three standard results from algebraic geometry:

\begin{fact}[\cite{Harris}]
	\label{fact:nested}
	The set of pairs of nested planes of fixed dimensions $k_1\le k_2$ in a
	vector space $V$,
	\begin{gather*}
		\Psi_{k_1,k_2}(V)
		:=\{(U_1,U_2)\colon U_1\subseteq U_2\}
		\subset G_{k_1}(V)\times G_{k_2}(V),
	\end{gather*}
	is a subvariety of the product of the corresponding Grassmannians. It
	comes with two projections onto the factors,
	\begin{gather*}
		G_{k_1}(V)\xleftarrow{\;\pi_1\;}\Psi_{k_1,k_2}(V)\xrightarrow{\;\pi_2\;}G_{k_2}(V).
	\end{gather*}
\end{fact}

\begin{fact}[\cite{Harris}]
	\label{fact:Fano}
	The set of $n$-planes contained in a given projective variety
	$X\subseteq\P^d$ forms a~sub\-variety in the Grassmannian of $n$-planes in
	$\P^d$, called the Fano variety $F_n(X)\subseteq G_n(\P^d)$ of~$X$.
\end{fact}

\begin{fact}
	\label{fact:commuting}
	The maximal commutative subalgebras consisting of symmetric endomorphisms
	of an $n$-dimensional vector space $V$ form a~subvariety $\mathcal
	C(V)\subseteq G_n(S^2V)$ in the Grassmannian of $n$-planes in~$S^2V$.
\end{fact}

\begin{proof}
	For def\/inite metrics the statement simply follows from the closed orbit
	lemma \cite{Borel}. Indeed, the space of diagonal endomorphisms def\/ines a
	point in $G_n(S^2V)$ and $\mathcal C(V)$ is the orbit of this point under
	the action of the algebraic group $\SOG(V)$. For indef\/inite metrics we
	have f\/initely many normal forms for maximal commuting subalgebras, each of
	them def\/ining a point in~$G_n(S^2V)$. By the closed orbit lemma, the
	orbit of each of these points def\/ines a~subvariety. Since~$\mathcal C(V)$
	is the union of these orbits, it is a~union of f\/initely many subvarieties
	and hence itself a subvariety.
\end{proof}

Our Main Theorem will be a simple consequence of the following lemma.

\begin{lemma}
	Let $X\subseteq G_n(V_1)$ and $Y\subseteq G_n(V_1\oplus V_2)$ be
	subvarieties. Then the set $Z=\{y\in Y\colon$ $\exists\, x\in X\colon p_1(y)\subseteq
	x\}$, where~$p_1$ is the projection onto the first factor, is a projective
	variety.
\end{lemma}

\begin{proof}
	First note that the condition $p_1(y)\subseteq x$ is equivalent to
	$y\subseteq p^{-1}(x)=x\oplus V_2$. Mapping $x\mapsto x\oplus V_2$ gives
	an embedding $\iota\colon G_n(V_1)\hookrightarrow G_m(V_1\oplus V_2)$, where
	$m=n+\dim V_2$. Indeed, taking $k_1=\dim V_2\le k_2=m$ and $V=V_1\oplus
	V_2$ in Fact~\ref{fact:nested}, the image is the variety
	$\pi_2\bigl(\pi_1^{-1}(\{0\oplus V_2\})\bigr)$. Applying
	Fact~\ref{fact:nested} once again, this time for $k_1=n\le k_2=m$, we
	obtain the variety $Z=\pi_1\bigl(\pi_2^{-1}(\iota(X))\bigr)\cap Y$.
\end{proof}

\begin{proof}[Proof of the Main Theorem]
	Fact~\ref{fact:Fano} implies that all $n$-planes contained in the variety
	$\mathcal I(M)$ $\subseteq\mathcal K(M)$ form a subvariety $F_n(\mathcal
	I(M))\subseteq G_n (\mathcal K(M) )$. These $n$-planes
	comprise St\"ackel systems, but not all of them are St\"ackel systems. We
	have to prove that the St\"ackel systems among them form a subvariety. This
	results from Fact~\ref{fact:commuting} as follows. We have already
	mentioned, that the Killing equation~\eqref{eq:Killing} is overdetermined
	of f\/inite type. To be more precise, all derivatives of a~Killing tensor
	$K_{ab}$ can be expressed as linear combinations of
	\begin{gather*}
		K_{ab} \in\Gamma(\,\yng(2)\,T^*\!M),\\
		\nabla_aK_{bc} \in\Gamma(\,\yng(2,1)\,T^*\!M),\\
		Y(\nabla_a\nabla_cK_{bd}) \in\Gamma(\,\yng(2,2)\,T^*\!M),
	\end{gather*}
	the coef\/f\/icients being polynomial in the Riemannian curvature tensor~\cite{Wolf}. Here
	\begin{itemize}\itemsep=0pt
		\item $\yng(2)\,T^*\!M=S^2T^*\!M$
			is the bundle of symmetric forms~$K_{ab}$ on~$M$;
		\item $\yng(2,1)\,T^*\!M$
			denotes the bundle of tensors~$T_{abc}$ on~$M$ with the symmetries
			\begin{gather*}
				T_{acb} =T_{abc},\qquad
				T_{abc}+T_{bca}+T_{cab} =0;
			\end{gather*}
		\item $\yng(2,2)\,T^*\!M$
			denotes the bundle of tensors $R_{abcd}$ with curvature tensor
			symmetry and~$Y$ stands for the corresponding Young projector
			(given by symmetrising in~$a$, $c$ and $b$, $d$, followed by an
			antisymmetrisation in $a$, $b$ and $c$, $d$).
	\end{itemize}
	This gives a linear evaluation map
	\begin{gather*}
		\operatorname{ev}_p\colon \
		\mathcal K(M)
		\longrightarrow
		\mathcal T_p
	\end{gather*}
	for each point $p\in M$, where
	\begin{gather*}
		\mathcal T_p:=
		\yng(2)\,T_p^*\!M\;\oplus\;
		\yng(2,1)\,T_p^*\!M\;\oplus\;
		\yng(2,2)\,T_p^*\!M
	\end{gather*}
	is the f\/ibre over $p$ of the tractor bundle $\mathcal TM$ corresponding to
	the prolongation of the Killing equation \eqref{eq:Killing}. This map is
	injective and gives an embedding
	\begin{gather}
		\label{eq:embedding}
		\iota\colon \ G_n\bigl(\mathcal K(M)\bigr)\hookrightarrow G_n(\mathcal T_p).
	\end{gather}
	On the other hand, by Fact~\ref{fact:commuting} we have an inclusion
	$\mathcal C(T_p^*\!M)\subseteq G_n(\yng(2)\,T_p^*\!M)$. We can now apply
	the above lemma to the varieties
	\begin{gather*}
		X =\mathcal C(T_p^*\!M)\subseteq G_n(\yng(2)\,T_p^*\!M),\qquad
		Y =\iota\bigl(F_n(\mathcal I(M))\bigr)\subseteq G_n(\mathcal T_p)
	\end{gather*}
	to obtain a variety $Z\subseteq G_n(\mathcal T_p)$. Its pull back
	$\iota^{-1}(Z)\subseteq G_n\bigl(\mathcal K(M)\bigr)$ under the embedding
	\eqref{eq:embedding} consists of all $n$-planes of Nijenhuis integrable
	Killing tensors whose members commute at the f\/ixed point~$p$. The
	intersection of these varieties for all $p\in M$ consists of $n$-planes of
	Nijenhuis integrable Killing tensors whose members commute at \emph{every}
	point~$p\in M$. By def\/inition, these are the St\"ackel systems on $M$.
	This proves the f\/irst part of the Main Theorem.

	By Observation~\ref{obs:groupaction} the variety $\mathcal S(M)$ is
	invariant under the induced action of the isometry group on
	$G_n(\mathcal K(M))$. This proves the second part.
\end{proof}

\section{The proof of concept}

Let $M=\S^n\subset\R^{n+1}$ denote the standard round sphere of dimension $n$.
It has the orthogonal group $\OG(n+1)$ as isometry group with Lie algebra
\begin{gather*}
	\mathfrak{so}(n+1)=\bigl\langle e_i\wedge e_j\colon 0\le i<j\le n\bigr\rangle.
\end{gather*}
Since the space of Killing tensors on constant sectional curvature spaces is
known to be generated by symmetric products of Killing vectors, we can embed
the space of Killing tensors into the space of symmetric forms on the Lie
algebra of the isometry group,
\begin{gather*}
	\mathcal K(\S^n)\subseteq S^2\mathfrak{so}(n+1).
\end{gather*}

\subsection{Normal forms}

While $\mathcal S(M)$ is a variety, this property will in general be lost for
the quotient $\mathcal S(M)/G$. As a remedy, one could consider the geometric
invariant theory quotient $\mathcal S(M)/\!/G$ instead or try to f\/ind a slice
for the $G$-action, i.e., a subvariety $\mathcal S_0(M)\subset\mathcal S(M)$
with f\/inite stabiliser which intersects every group orbit transversally. In
the latter case we would still deal with a variety which parametrises isometry
classes of separable coordinates, but trade this of\/f against having a unique
parametrisation. This is acceptable, since the ambiguity is only f\/inite. In
the case of spheres, such a slice is provided by the following result.

\begin{theorem}[\cite{Boyer&Kalnins&Winternitz}]
	Up to an isometry every St\"ackel system is contained in
	\begin{gather*}
		\mathcal K_0(S^n):=\bigl\langle(e_i\wedge e_j)^2\colon 0\le i<j\le n\bigr\rangle.
	\end{gather*}
\end{theorem}

Note that $\mathcal K_0(S^n)$ is the space of diagonal symmetric forms on the
Lie algebra $\mathfrak{so}(n+1)$ with respect to the basis $\{e_i\wedge
e_j\}$.

\begin{definition}
	We call $\mathcal K_0(S^n)$ the subspace of Killing tensors on $\S^n$
	\emph{in normal form} and def\/ine
	\begin{gather*}
		\mathcal S_0(\S^n):=
		\mathcal S(\S^n)\cap G_n\bigl(\mathcal K_0(\S^n)\bigr)
	\end{gather*}
	to be the \emph{variety of separable coordinates in normal form}.
\end{definition}

It is not dif\/f\/icult to f\/igure out that the stabiliser of $\mathcal K_0(S^n)$
in $\OG(n+1)$ is the hyperoctahedral group $S_{n+1}\ltimes{\mathbb
Z_2}^{n+1}$, which acts by permutations and sign changes of the f\/ixed basis
$\{e_i\}$, and that the stabiliser action on $\mathcal K_0(S^n)$ descends to
the permutation group $S_{n+1}$. We therefore have
\begin{gather}
	\label{eq:slicing}
	\frac{\mathcal S(\S^n)}{\OG(n+1)}\isom\frac{\mathcal S_0(\S^n)}{S_{n+1}}.
\end{gather}
That is, we replace the quotient of a variety by a continuous group with a
quotient of a linear section of this variety by a f\/inite group.

\subsection{The isomorphism}

Recall our Main Problem, to describe the structure of the variety $\mathcal
S(M)$ and the isometry group action on it. In the preceding section we
reduced this problem for $M=\S^n$ to describing the variety $\mathcal
S_0(\S^n)$ and the $S_{n+1}$-action on it. The following theorem solves our
Main Problem at once, by identifying $\mathcal S_0(\S^n)$ to a
complicated, but well understood moduli space.

\medskip
\begin{flushright}
	``In great mathematics \\
	there is a very high degree of unexpectedness, \\
	combined with inevitability and economy.'' \\
	\rule{.48\textwidth}{0.5pt}\\
	\textsc{Godfrey Harold Hardy}
\end{flushright}
\medskip

\begin{theorem}[\cite{Schoebel&Veselov}]
	\label{thm:iso}
	The variety of St\"ackel systems on $\S^n$ in normal form is a smooth $(!)$
	projective variety, isomorphic to the real part of the Deligne--Mumford
	moduli space of stable curves of genus~$0$ with~$n+2$ marked points:
	\begin{gather}
		\label{eq:iso}
		\mathcal S_0(\S^n)\isom\bar{\mathscr M}_{0,n+2}(\R).
	\end{gather}
	This isomorphism is $S_{n+1}$-equivariant.
\end{theorem}

To give this theorem a meaning, one would have to say some words about the
moduli spa\-ces~$\bar{\mathscr M}_{0,n}$ and their real parts. We leave it at
saying that $\bar{\mathscr M}_{0,n+1}(\R)$ is a compactif\/ication of the
conf\/iguration space of~$n+1$ distinct points on~$\P^1$ modulo automorphisms
or, equivalently, of~$n$ points on the real line modulo the af\/f\/ine group.
Giving a formal def\/inition would go far beyond the scope of this notice and we
refer to the literature for more details on moduli spaces of curves
\cite{Deligne&Mumford,Harris&Morrison,Knudsen,Knudsen+,Mumford}. To get an idea how
these spaces look like, let us list the simplest examples.
\begin{itemize}\itemsep=0pt
	\item $\bar{\mathscr M}_{0,3}(\R)$ is a point.
	\item $\bar{\mathscr M}_{0,4}(\R)$ is a projective line.
	\item $\bar{\mathscr M}_{0,5}(\R)$ is the blowup of a projective plane in four points.
	\item In general, $\bar{\mathscr M}_{0,n+2}(\R)$ is an iterated blowup of $\P^{n-1}$ \cite{Davis&Januszkiewicz&Scott,Devadoss,Kapranov}.
\end{itemize}

The isomorphism in Theorem~\ref{thm:iso} is rather surprising, as it relates
two seemingly completely unrelated objects~-- separable coordinates for the
Hamilton--Jacobi equation on one hand and stable algebraic curves with marked
points on the other. We can exploit the isomorphism~\eqref{eq:iso} by
``pulling back'' known properties from $\bar{\mathscr M}_{0,n+2}(\R)$ to
$\mathcal S_0(\S^n)$ and interpret them in terms of separable coordinates on
spheres. In the rest of this section we will illustrate this by giving three
examples of properties which reveal hitherto unknown structures in the long
known classif\/ication of orthogonal separable coordinates.

\subsection{Stashef\/f polytopes}

A property which can readily be used to describe the quotient $\mathcal
S(\S^n)/\OG(n+1)$ is the fact that the moduli space $\bar{\mathscr
M}_{0,n+2}(\R)$ is tiled by $(n+1)!/2$ copies of the Stashef\/f polytope
$K_{n+1}$ \cite{Devadoss, Kapranov}. \emph{Stasheff polytopes}, or
\emph{associahedra}, were introduced by Stashef\/f as a combinatorial object in
the homotopy theory of $H$-spaces \cite{Stasheff,Stasheff+}\footnote{See also the historical note in Stashef\/f's contribution to~\cite{Tamari}.}.

Instead of a formal def\/inition, we list the simplest examples, which will be
suf\/f\/icient for our purpose of illustration:
\begin{itemize}\itemsep=0pt
	\item $K_2$ is a point;
	\item $K_3$ is a line segment;
	\item $K_4$ is a pentagon;
	\item in general, $K_{n+1}$ is an ($n-1$)-dimensional polytope.
\end{itemize}

Combining the homeomorphism \eqref{eq:slicing} with the equivariant isomorphism~\eqref{eq:iso}, we obtain
\begin{gather}
	\label{eq:quotient}
	\frac{\mathcal S(\S^n)}{\OG(n+1)}
	\isom\frac{\bar{\mathscr M}_{0,n+2}(\R)}{S_{n+1}}.
\end{gather}
The $S_{n+1}$-action on $\bar{\mathscr M}_{0,n+2}(\R)$ is transitive on the
$(n+1)!/2$ tiles and the stabiliser of a single tile is $\Z_2$, a ref\/lection
symmetry of the Stashef\/f polytope $K_{n+1}$. A fundamental domain for this
action is therefore ``one half'' of the Stashef\/f polytope~$K_{n+1}$ and the
quotient~\eqref{eq:quotient} is obtained by identifying certain faces of
$K_{n+1}/\Z_2$. In the next section we give the exact recipe which faces have
to be identif\/ied.

\subsection{Example}

On $\S^2$ there are two families of separable coordinates, as shown in
Fig.~\ref{fig:sphelliptic}: One is the familiar spherical coordinates,
given by longitude and latitude. The other is elliptic coordinates,
consisting of two families of confocal ellipses around two pairs of antipodal
focal points. Modulo isometries, they form a one-parameter family,
parametrised by the angular distance between the foci. Obviously, elliptic
coordinates degenerate to spherical coordinates if this angle goes to~$0$ or~$\pi$.

\begin{figure}[t]
\centering
	\includegraphics[width=.35\textwidth]{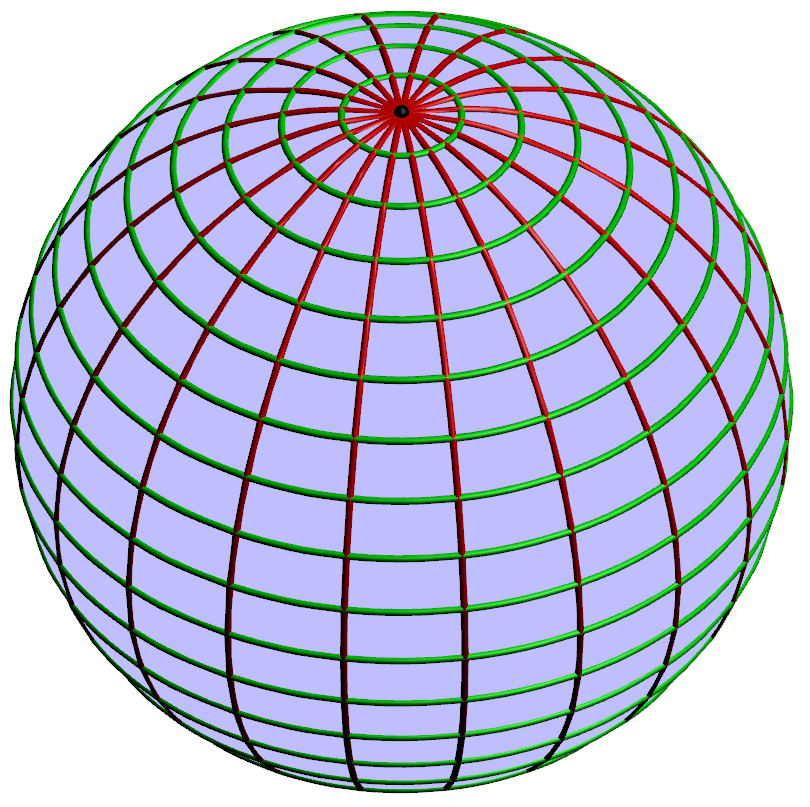}\qquad
	\includegraphics[width=.35\textwidth]{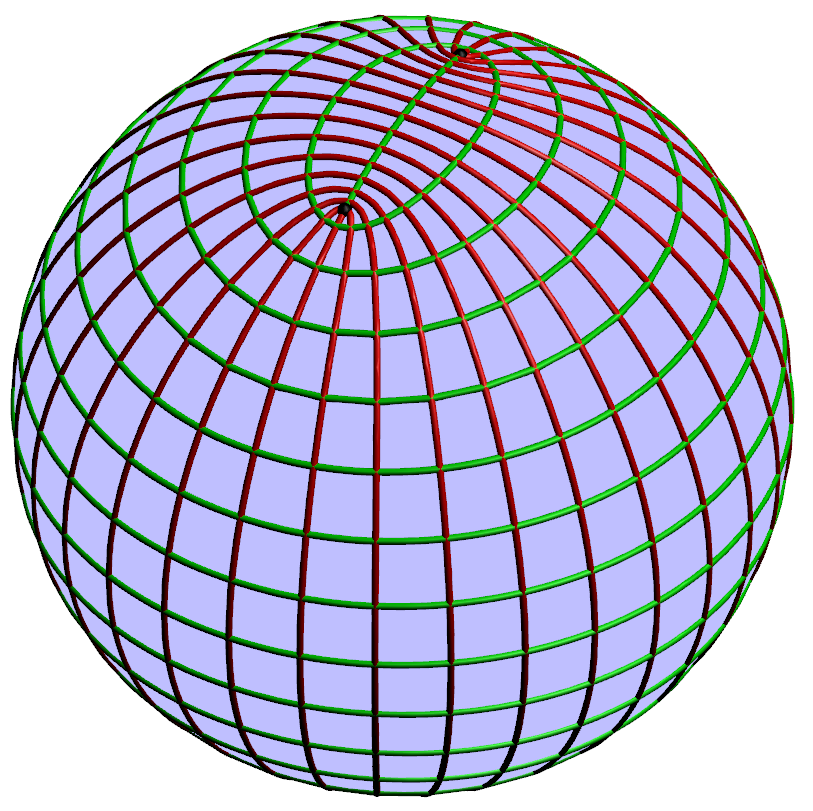}

	\caption{Spherical and elliptic coordinates on~$\S^2$.}	\label{fig:sphelliptic}
\end{figure}

Fig.~\ref{fig:S2} shows how these two families are assembled to give the
moduli space $\bar{\mathscr M}_{0,4}(\R)$ tiled by three copies of~$K_3$,
which is a circle tiled by three intervals. The interiors of the three
intervals correspond to elliptic coordinates whose focal points lie in one of
the three coordinate planes, while the three boundary points separating them
correspond to spherical coordinates with the poles on one of the three
coordinate axes. When we cross the boundary between two intervals, the two
pairs of foci f\/irst coalesce into one pair of poles, and then separate in
orthogonal direction. Note that this trajectory is not smooth, which is why
the smoothness of~$\mathcal S_0(\S^2)$ comes rather unexpected.

\begin{figure}[h]
	\centering
	\includegraphics[width=.6\textwidth]{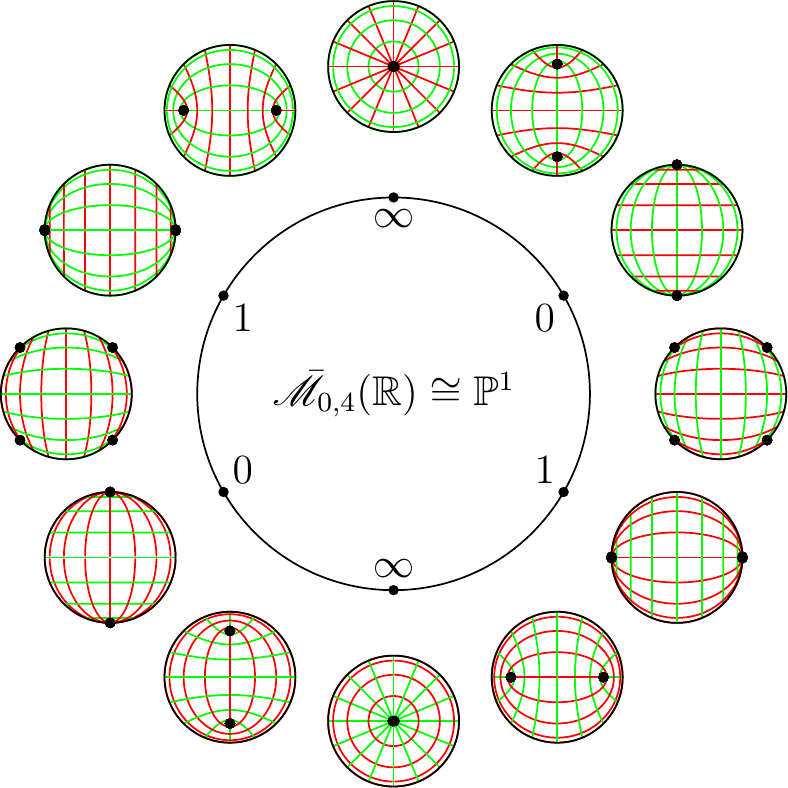}
	\caption{Parametrisation of separable coordinates in normal form on $\S^2$ by
		$\bar{\mathscr M}_{0,4}(\R)\isom\P^1$. To indicate the monodromy, a
		two-fold covering of the projective line is shown.}
	\label{fig:S2}
\end{figure}

Note also that there is a monodromy if we consider coordinates as ordered: If
we start with spherical coordinates and track the longitudes while going once
around the circle, we end up with latitudes and vice versa. This is
illustrated in Fig.~\ref{fig:S2} by showing a two-fold cover (which is again
a circle) and using a dif\/ferent colour for each coordinate.

The quotient $\mathcal S_0(\S^2)/S_3$ is an interval of which one endpoint
corresponds to spherical coordinates, while the other one corresponds to
elliptic coordinates where the four focal points are equidistantly spaced. In
particular, these points are non-isometric and hence not identif\/ied in~$K_3/\Z_2$.

Dimension $n=2$ is, however, trivial from our point of view, for in this case
the Nijenhuis integrability conditions~\eqref{eq:Nijenhuis} are void, owing to
the antisymmetrisation over three indices. For $n=3$ the tiling of
$\bar{\mathscr M}_{0,5}(\R)$ by twelve copies of $K_4$ can be seen as follows.
Four general points in the projective plane determine six projective lines.
These lines divide the plane into twelve triangles, each having two of the
four points as vertices. A blowup in these four points therefore transforms
the triangles into pentagons, which form the twelve tiles of~$\bar{\mathscr
M}_{0,5}(\R)$.

\subsection{Rooted planar trees}

While points in $\mathcal S_0(\S^n)$ parametrise single isometry classes of
separable coordinates, the (open) faces of a Stashef\/f polytope provide a more
coarse classif\/ication into families. Stashef\/f polytopes have a number of
remarkable combinatorial properties \cite{Devadoss,Devadoss&Read}, which can
readily be used to label and count these families of separable coordinates.
\begin{itemize}\itemsep=0pt
	\item
		The faces of the Stashef\/f polytope $K_n$ can be labelled by rooted
		planar trees with $n$ leaves.
	\item
		The codimension $m$ faces correspond to trees with $m$ inner nodes
		other than the root. In particular, the number of vertices of $K_n$
		is $C_{n-1}$, where $C_n=\frac1{n+1}\binom{2n}n$ is the \emph{Catalan
		number}.
\end{itemize}

We can therefore label the dif\/ferent families of orthogonal separable
coordinates on spheres by rooted planar trees. For instance:

\begin{itemize}\itemsep=0pt
	\item
		Elliptic coordinates on $\S^n$, introduced by C.~Neumann
		\cite{Neumann}, form a family with $n-1$ continuous parameters,
		labelled by a corolla tree with $n+1$ leaves, as shown in
		Fig.~\ref{fig:trees} for $n+1=5$. They label the interior of the
		Stashef\/f polytope.
\begin{figure}[h]
			\centering
\begin{minipage}{45mm}\centering
			\includegraphics[width=40mm]{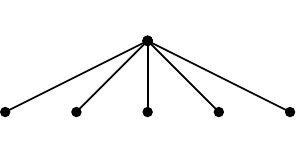}\\
(a) corolla tree.
\end{minipage}\qquad
\begin{minipage}{45mm}\centering
\includegraphics[width=40mm]{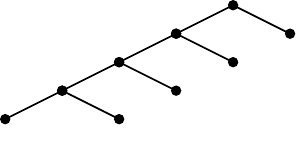} \\ (b) left comb tree.
\end{minipage}
			\caption{}	\label{fig:trees}
		\end{figure}
	\item
		Polyspherical coordinates, introduced by Vilenkin \cite{Vilenkin},
		form a discrete family labelled by binary trees
		\cite{Vilenkin&Klimyk}. They correspond to the vertices of the
		Stashef\/f polytope.
	\item
		In particular, standard spherical coordinates correspond to a left
		comb tree, as shown in Fig.~\ref{fig:trees}.
\end{itemize}

Two trees which dif\/fer by reversing the order of a node's children or by a
sequence of such opera\-tions are called \emph{dyslectic}. Separation
coordinates are isometric if and only if the corresponding trees are
dyslectic. This gives the recipe how to obtain the quotient
\eqref{eq:quotient} from $K_{n+1}/\Z_2$: One has to identify those faces
labelled by dyslectic trees.

For example, Fig.~\ref{fig:K4} shows the Stashef\/f polytope $K_4$ with its
faces labelled by trees and the corresponding families of separable
coordinates. This matches the classical results obtained by
Eisenhart~\cite{Eisenhart} and Olevski\u\i~\cite{Olevskii}. The ref\/lection
symmetry divides $K_4$ into two halves, each a~quadrilateral. Two adjacent
vertices of this quadrilateral are labelled by dyslectic trees and hence have
to be identif\/ied. One of them is a comb tree, so these are isometric to
standard spherical coordinates.

\begin{figure}[h]
	\centering
	\includegraphics[height=0.44\textwidth]{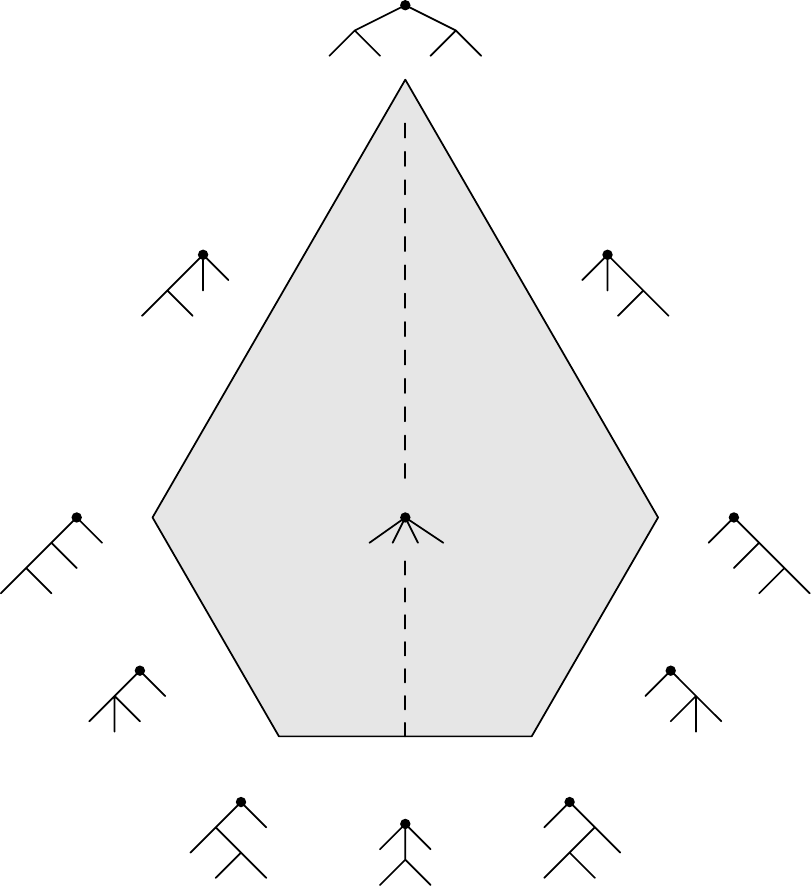}\qquad {\includegraphics[height=.44\textwidth]{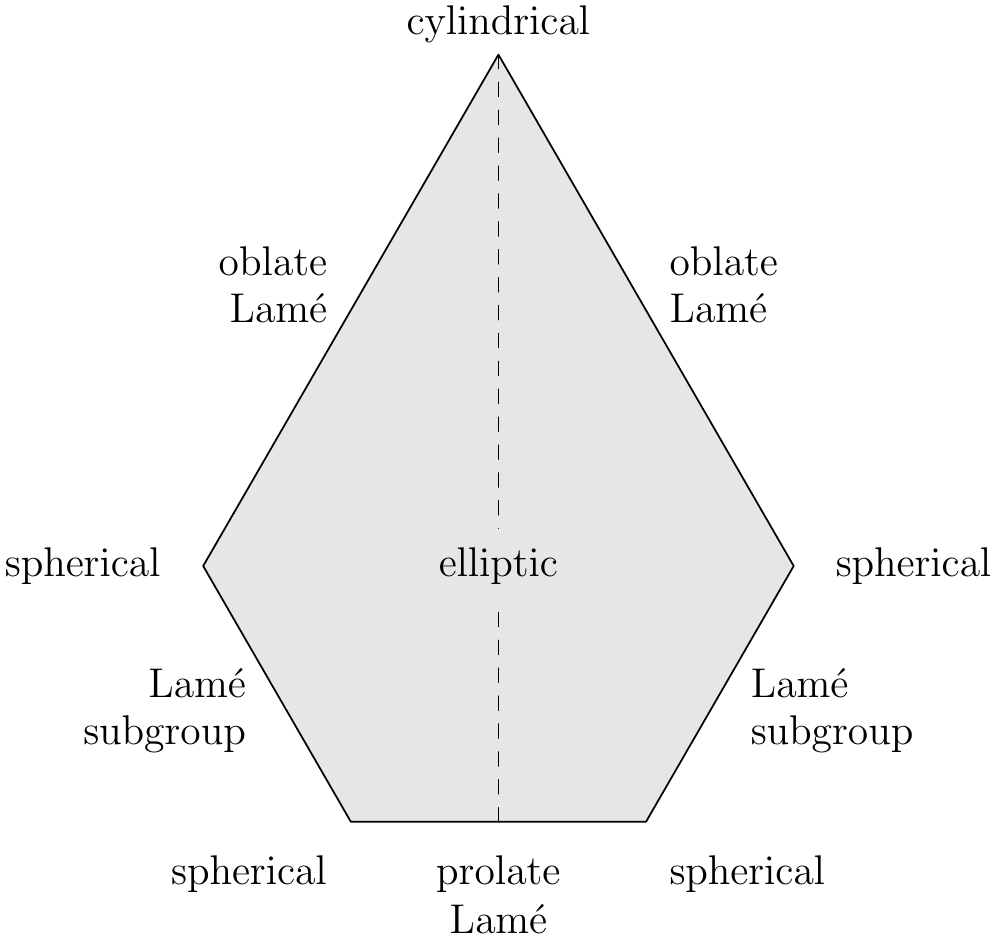}}
	\caption{The faces of the Stashef\/f polytope $K_4$, labelled by rooted trees
		(left) and separable coordinates on~$\S^3$ (right).}
	\label{fig:K4}
\end{figure}

In this way we can not only explain intrinsically the families in~\cite{Kalnins&Miller} and why they can be labelled by trees. The Stashef\/f
polytope also retains the information on how the individual families
degenerate into each other. It is, for instance, obvious in our setting that
all orthogonal separable coordinates are limits of elliptic coordinates.

\subsection{Operad structure}

A remarkable fact about the moduli spaces $\bar{\mathscr M}_{0,n}(\R)$, which
will allow us to give a simple way how to explicitly construct separable
coordinates, is that the sequence $\mathcal O(n):=\bar{\mathscr
M}_{0,n+1}(\R)$ carries a~natural operad structure, called the \emph{mosaic
operad}~\cite{Devadoss}.

An operad is best explained with composition of functions. Given a function
$f$ depending on~$k$ arguments and $k$ functions $f_1,\ldots,f_k$, each
depending on an arbitrary (f\/inite) number~$n_i$ of arguments, we can compose
them to give a function $f\circ(f_1,\ldots,f_k)$ def\/ined by
\begin{gather*}
	\bigl(f\circ(f_1,\ldots,f_k)\bigr)(\boldsymbol x_1,\ldots,\boldsymbol x_k)
	:=
	f\bigl(f_1(\boldsymbol x_1),\ldots,f_k(\boldsymbol x_k)\bigr)
\end{gather*}
and depending on $n_1+\cdots+n_k$ arguments. An \emph{operad} is an algebraic
object formalising the properties of this composition: left/right identity,
associativity and equivariance.
\begin{definition}
	An \emph{operad} is a sequence of sets $\mathcal O(n)$ indexed by
	$n=1,2,3,\ldots$ together with a~composition map
	\begin{gather*}
		\begin{array}{@{}rrcl}
			\circ\colon &\mathcal O(k)\times\mathcal O(n_1)\times\cdots\times\mathcal O(n_k)&\longrightarrow&\mathcal O(n_1+\cdots+n_k),\\
			&(y,x_1,\ldots,x_k)&\mapsto&y\circ(x_1,\ldots,x_k)
		\end{array}
	\end{gather*}
	and a right action
	\begin{gather*}
		\begin{array}{@{}rrcl}
			\star\colon &\mathcal O(n)\times S_n&\longrightarrow&\mathcal O(n),\\
			&(x,\pi)&\mapsto&x\star\pi
		\end{array}
	\end{gather*}
	of the permutation group $S_n$, satisfying the following properties:
	\begin{description}\itemsep=0pt
		\item[Identity:]
			There is a distinguished element $1\in\mathcal O(1)$ with
			\begin{gather*}
				y\circ(1,\ldots,1)=y=1\circ y.
			\end{gather*}
		\item[Associativity:]
			\begin{gather*}
				z\circ(y_1\circ x_1,\ldots,y_k\circ x_k)=\bigl(z\circ(y_1,\ldots,y_k)\bigr)\circ(x_1,\ldots,x_k).
			\end{gather*}
		\item[Equivariance:]
			\begin{gather*}
				(y\star\pi)\circ\bigl((x_1,\ldots,x_k)\star\pi\bigr) =\bigl(y\circ(x_1,\ldots,x_k)\bigr)\star\pi,\\
				y\circ(x_1\star\pi_1,\ldots,x_k\star\pi_k) =\bigl(y\circ(x_1,\ldots,x_k)\bigr)\star(\pi_1,\ldots,\pi_k),
			\end{gather*}
			where $S_k$ acts on $(x_1,\ldots,x_k)$ by permutation and
			$(\pi_1,\ldots,\pi_k)$ on $\mathcal O(n_1+\cdots+n_k)$ under the
			inclusion $S_{n_1}\times\cdots\times S_{n_k}\hookrightarrow
			S_{n_1+\cdots+n_k}$.
	\end{description}
\end{definition}

We do not need here the explicit def\/inition of the mosaic operad on the moduli
spaces $\bar{\mathscr M}_{0,n+1}(\R)$. Theorem~\ref{thm:iso} tells us that
there is such a structure on $\mathcal O(n):=\mathcal S_0(\S^{n-1})$ as
well and once knowing it is there it is not too dif\/f\/icult to f\/igure it out in
terms of separable coordinates. Indeed, the sequence of spheres $\mathcal
O(n)=S^{n-1}$ carries a simple operad structure, given by the composition
\begin{gather}
	\label{eq:composition}
	\begin{array}{rcl}
		S^{k-1}\times S^{n_1-1}\times\cdots\times S^{n_k-1}&\longrightarrow&S^{n_1+\cdots+n_k-1},\\
		(\boldsymbol y,\boldsymbol x_1,\ldots,\boldsymbol x_k)&\mapsto&
		\boldsymbol y\circ(\boldsymbol x_1,\ldots,\boldsymbol x_k)
		:=
		(y_1\boldsymbol x_1,\ldots,y_k\boldsymbol x_k).
	\end{array}
\end{gather}
By pullback, this def\/ines an operad on (local) coordinate systems. It turns
out that this operad restricts to orthogonal separable coordinate systems in
normal form, which gives the translation of the mosaic operad from
$\bar{\mathscr M}_{0,n+1}(\R)$ to $\mathcal S_0(\S^{n-1})$. Moreover, the
mosaic operad can also be expressed in a simple way in terms of St\"ackel
systems \cite{Schoebel&Veselov}.

The mosaic operad induces an operad on rooted planar trees, whose composition
$T\circ(T_1$, $\ldots,T_k)$ is given by grafting the $k$ trees $T_1, \ldots,T_k$
with their respective roots to the~$k$ leaves of the tree~$T$. Now observe
that by iterating this composition any rooted planar tree can be constructed
from corolla trees such as the one shown in Fig.~\ref{fig:trees}, which
correspond to elliptic coordinates. Consequently, on spheres all orthogonal
separation coordinates in a certain dimension can be constructed by composing
elliptic coordinates on lower-dimensional spheres under the operad composition~\eqref{eq:composition}.

Classically, non-generic separable coordinates have been studied via limits of
elliptic coordinates, see for example~\cite{Kalnins&Miller&Reid}. This
requires an intricate analysis, since these limits strongly depend on the path
chosen to approach the limit. In our context this is obvious, because the
moduli spaces $\bar{\mathscr M}_{0,n}(\R)$ are iterated blowups. Our result
replaces the complicated limiting procedure for elliptic coordinates in a
f\/ixed dimension by a simple operad composition of elliptic coordinates from
lower dimensions. Likewise, it gives a powerful recursive method for proving
statements about separable coordinates and related objects like St\"ackel
systems: First prove that the statement holds for generic coordinates and then
show that it is stable under the operad composition.

\section{Open problems}

From our algebraic geometric point of view on the classif\/ication of separable
coordinates in conjunction with the example in the preceding section arise a~number of natural questions, which we would like to address here.

\medskip
\begin{flushright}
	\small
	``The art of doing mathematics \\
	consists in f\/inding that special case \\
	which contains all the germs of generality.'' \\
	\rule{.42\textwidth}{0.5pt}\\
	\textsc{David Hilbert}
\end{flushright}
\medskip

The proof of Theorem~\ref{thm:iso} is unsatisfactory in the sense that it does
not give an explicit construction of stable algebraic curves with marked
points from separable coordinates on spheres, or vice versa. Such a~construction would be desirable in view of generalisations of this result (see
below).

\begin{problem}
	Give a direct construction of $($equivalence classes of$)$ stable algebraic
	curves with marked points from $($equivalence classes of$)$ orthogonal
	separable coordinates on spheres $($or vice versa$)$, realising the
	isomorphism~\eqref{eq:iso}.
\end{problem}

Finding separable coordinates is only the f\/irst step in solving the partial
dif\/ferential equation under consideration. The second is to actually perform
the separation of variables. This reduces the partial dif\/ferential equation
in $n$ variables to $n$ ordinary dif\/ferential equations. The third step is to
solve these ordinary dif\/ferential equations individually (and the fourth to
put them together to an explicit solution of the initial equation). The
operad structure we found for separable coordinates should therefore be
somehow ref\/lected on the corresponding special functions as well.

\begin{problem}
	Find an explicit operad structure on those special functions that arise
	from a~separation of variables on spheres.
\end{problem}

Separable coordinates on Euclidean space can be obtained by suitable
Inonu-Wigner contractions of separable coordinates on the sphere~\cite{IPSW}.
In the spirit of our approach, this construction should have an algebraic
geometric counterpart, a~``contraction of varieties'' from $\mathcal S(\S^n)$
to $\mathcal S(\mathbb E^n)$.

\begin{problem}
	Give an algebraic geometric interpretation of In\"on\"u--Wigner contractions of
	St\"ackel systems on spheres.
\end{problem}

A posteriori, one can f\/ind some clues for the mosaic operad on separable
coordinates on spheres in the work of Kalnins and Miller, namely the
composition formulas (2.29), (2.45), (2.55) and (3.14) in
\cite{Kalnins&Miller} or (3.45), (3.61) and (3.71) in \cite{Kalnins}. For
f\/lat and hyperbolic space similar, but more complicated composition formulas
are given there as well, cf.\ formulas (5.17), (4.25) in
\cite{Kalnins&Miller} as well as (5.39), (5.42), (5.47), (5.52), (5.54),
(5.60) and (5.62) in~\cite{Kalnins}. They involve separable coordinates on
lower-dimensional constant curvature spaces which are not necessarily of the
same curvature. This indicates that there might be an operad structure, or
rather a generalisation thereof, for orthogonal separable coordinates on
arbitrary constant curvature spaces.

\begin{problem}
	Is there a $($generalised$)$ operad on separable coordinates on constant
	curvature manifolds?
\end{problem}

Here we have only considered classical separation of variables, but there
exist several f\/lavours of it. In the complex setting one works on a complex
instead of a real Riemannian manifold~\cite{Kalnins&Miller&Reid}. In the
conformal setting Killing tensors are replaced by conformal Killing tensors~\cite{Kalnins&Miller83}. Our approach adapts to both of these settings,
yielding complex and conformal analogues of the variety~$\mathcal S(M)$. For
constant curvature spaces, they all contain~$\mathcal S(\S^n)$ as a subvariety
and therefore extend the moduli spaces $\bar{\mathscr M}_{0,n+2}(\R)$.

\begin{problem}
	Are the varieties $\mathcal S(M)$ for constant curvature manifolds~$M$ and
	their complex or conformal analogues related to moduli spaces of
	algebro-geometric objects generalising stable algebraic curves of genus
	zero with marked points?
\end{problem}

Another f\/lavour is non-orthogonal separation of variables, for which it is not
even clear whether our approach extends to this case as well or not.

\begin{problem}
	Does the set of separable coordinates $($orthogonal or not$)$ carry a natural
	structure of an algebraic variety?
\end{problem}

Our Main Theorem holds for any (pseudo-)Riemannian manifold, but all
interesting examples we are aware of are for constant curvature manifolds.
Trivial examples are manifolds with no or just a single St\"ackel system,
manifolds of dimension two, for which the Nijenhuis conditions~\eqref{eq:Nijenhuis} are void so that $\mathcal S(M)\isom\P^{d-2}$ with
$d=\dim\mathcal K(M)$, or products $M_1\times M_2$, for which the inclusion
$\mathcal S(M_1\times M_2)\subseteq\mathcal S(M_1)\times\mathcal S(M_2)$ is an
equality. Possible candidates for non-trivial examples include the Thurston
geometries in dimension three and homogeneous spaces, as each of these classes
comprises spheres.

\begin{problem}
	Find a non-constant curvature manifold $M$ for which the variety~$\mathcal
	S(M)$ has a~non-trivial geometry.
\end{problem}

The key observation which implied that variable separation is an algebraic
geometric problem was that we seek for solutions $K$ of the Nijenhuis
integrability conditions within a f\/inite-dimensional vector space and that
these equations are algebraic in $K$ and its derivatives. This means that our
proposed approach is not bound to separable systems, but applies to any
problem whose determining equations can be cast into this form. Another
example of such a problem which is ``covertly algebraic geometric'' has
recently been found by Jonathan Kress and the author~\cite{Kress&Schoebel}, by
showing that the set of superintegrable systems in the plane likewise carries
the structure of a projective variety equipped with an isometry group action.
Here, too, the study of the geometry of this variety reveals hitherto unknown
structure in the known classif\/ication. We prove, for example, that to every superintegrable system in the plane
there is a~canonically associated planar arrangement of line triples. So
for superintegrable systems hyperplane arrangements seem to play the same
role as stable algebraic curves with marked points for separable systems.

\begin{problem}
	Are there further $($classification$)$ problems which are ``covertly algebraic
	geometric''? By this we mean that the determining equations are partial
	differential equations which are polynomial in the variables and their
	derivatives and for which one seeks all solutions within a~finite-dimensional vector space.
\end{problem}

\subsection*{Acknowledgements}

This notice is based on a talk held at the workshop ``Analytical Mechanics and
Dif\/ferential Geometry'' at the Universit\`a di Torino on 12th and 13th March
2015 on the occasion of Sergio Benenti's 70th birthday. The author would like
to thank the organisers, Claudia Chanu and Giovanni Rastelli, for their kind
invitation and hospitality, as well as Willard Miller for valuable discussions
on the subject.

\pdfbookmark[1]{References}{ref}
\LastPageEnding


\begin{thebibliography}{99}
\footnotesize\itemsep=0pt

\bibitem{Benenti}
Benenti S., Orthogonal separable dynamical systems, in Dif\/ferential Geometry
 and its Applications ({O}pava, 1992), \textit{Math. Publ.}, Vol.~1, Editors
 O.~Kowalsky, D.~Krupka, Silesian University Opava, Opava, 1993, 163--184.

\bibitem{Borel}
Borel A., Linear algebraic groups, \href{http://dx.doi.org/10.1007/978-1-4612-0941-6}{\textit{Graduate Texts in Mathematics}}, Vol.~126, Springer, New York, 1991.

\bibitem{Boyer&Kalnins&Winternitz}
Boyer C.P., Kalnins E.G., Winternitz P., Separation of variables for the
 {H}amilton--{J}acobi equation on complex projective spaces, \href{http://dx.doi.org/10.1137/0516006}{\textit{SIAM~J.
 Math. Anal.}} \textbf{16} (1985), 93--109.

\bibitem{Davis&Januszkiewicz&Scott}
Davis M., Januszkiewicz T., Scott R., Nonpositive curvature of blow-ups,
 \href{http://dx.doi.org/10.1007/s000290050039}{\textit{Selecta Math.~(N.S.)}} \textbf{4} (1998), 491--547.

\bibitem{Deligne&Mumford}
Deligne P., Mumford D., The irreducibility of the space of curves of given
 genus, \href{http://dx.doi.org/10.1007/BF02684599}{\textit{Inst. Hautes \'Etudes Sci. Publ. Math.}} \textbf{36} (1969), 75--109.

\bibitem{Devadoss}
Devadoss S.L., Tessellations of moduli spaces and the mosaic operad, in
 Homotopy Invariant Algebraic Structures ({B}altimore, {MD}, 1998),
 \href{http://dx.doi.org/10.1090/conm/239/03599}{\textit{Contemp. Math.}}, Vol.~239, Amer. Math. Soc., Providence, RI, 1999,
 91--114, \href{http://arxiv.org/abs/math.AG/9807010}{math.AG/9807010}.

\bibitem{Devadoss&Read}
Devadoss S.L., Read R.C., Cellular structures determined by polygons and trees,
 \href{http://dx.doi.org/10.1007/PL00001293}{\textit{Ann. Comb.}} \textbf{5} (2001), 71--98, \href{http://arxiv.org/abs/math.CO/0008145}{math.CO/0008145}.

\bibitem{Eisenhart}
Eisenhart L.P., Separable systems of {S}t\"ackel, \href{http://dx.doi.org/10.2307/1968433}{\textit{Ann. of Math.}}
 \textbf{35} (1934), 284--305.

\bibitem{Harris}
Harris J., Algebraic geometry, \href{http://dx.doi.org/10.1007/978-1-4757-2189-8}{\textit{Graduate Texts in Mathematics}}, Vol.~133, Springer-Verlag, New York, 1992.

\bibitem{Harris&Morrison}
Harris J., Morrison I., Moduli of curves, \href{http://dx.doi.org/10.1007/b98867}{\textit{Graduate Texts in
 Mathematics}}, Vol.~187, Springer-Verlag, New York, 1998.

\bibitem{Horwood&McLenaghan&Smirnov}
Horwood J.T., McLenaghan R.G., Smirnov R.G., Invariant classif\/ication of
 orthogonally separable Hamiltonian systems in Euclidean space, \href{http://dx.doi.org/10.1007/s00220-005-1331-8}{\textit{Comm.
 Math. Phys.}} \textbf{259} (2005), 670--709, \href{http://arxiv.org/abs/math-ph/0605023}{math-ph/0605023}.

\bibitem{IPSW}
Izmest'ev A.A., Pogosyan G.S., Sissakian A.N., Winternitz P., Contractions of
 {L}ie algebras and separation of variables. {T}he {$n$}-dimensional sphere,
 \href{http://dx.doi.org/10.1063/1.532820}{\textit{J.~Math. Phys.}} \textbf{40} (1999), 1549--1573.

\bibitem{Kalnins}
Kalnins E.G., Separation of variables for {R}iemannian spaces of constant
 curvature, \textit{Pitman Monographs and Surveys in Pure and Applied
 Mathematics}, Vol.~28, Longman Scientif\/ic \& Technical, Harlow, John Wiley \&
 Sons, Inc., New York, 1986.

\bibitem{Kalnins&Miller83}
Kalnins E.G., Miller Jr. W., Conformal {K}illing tensors and variable
 separation for {H}amilton--{J}acobi equations, \href{http://dx.doi.org/10.1137/0514009}{\textit{SIAM~J. Math. Anal.}}
 \textbf{14} (1983), 126--137.

\bibitem{Kalnins&Miller}
Kalnins E.G., Miller Jr. W., Separation of variables on {$n$}-dimensional
 {R}iemannian manifolds. {I}.~{T}he {$n$}-sphere {$S_n$} and {E}uclidean
 {$n$}-space {${\bf R}^n$}, \href{http://dx.doi.org/10.1063/1.527088}{\textit{J.~Math. Phys.}} \textbf{27} (1986),
 1721--1736.

\bibitem{Kalnins&Miller&Reid}
Kalnins E.G., Miller Jr. W., Reid G.J., Separation of variables for complex
 {R}iemannian spaces of constant curvature. {I}.~{O}rthogonal separable
 coordinates for {${\rm S}_{n{\bf C}}$} and {${\rm E}_{n{\bf C}}$},
 \href{http://dx.doi.org/10.1098/rspa.1984.0075}{\textit{Proc. Roy. Soc. London Ser.~A}} \textbf{394} (1984), 183--206.

\bibitem{Kapranov}
Kapranov M.M., The permutoassociahedron, {M}ac {L}ane's coherence theorem and
 asymptotic zones for the {KZ} equation, \href{http://dx.doi.org/10.1016/0022-4049(93)90049-Y}{\textit{J.~Pure Appl. Algebra}}
 \textbf{85} (1993), 119--142.

\bibitem{Knudsen}
Knudsen F.F., The projectivity of the moduli space of stable curves.
 {II}.~{T}he stacks {$M_{g,n}$}, \textit{Math. Scand.} \textbf{52} (1983),
 161--199.

\bibitem{Knudsen+}
Knudsen F.F., The projectivity of the moduli space of stable curves. {III}.~The
 line bundles on $M_{g,n}$, and a proof of the projectivity of $\overline
 M_{g,n}$ in characteristic~$0$, \textit{Math. Scand.} \textbf{52} (1983),
 200--212.

\bibitem{Kress&Schoebel}
Kress J., Sch\"obel K., An algebraic geometric classif\/ication of
 superintegrable systems on the {E}uclidean plane, \href{http://arxiv.org/abs/1602.07890}{arXiv:1602.07890}.

\bibitem{Tamari}
M{\"u}ller-Hoissen F., Pallo J.M., Stashef\/f J. (Editors), Associahedra,
 {T}amari lattices and related structures. Tamari memorial Festschrift,
 \href{http://dx.doi.org/10.1007/978-3-0348-0405-9}{\textit{Progress in Mathematical Physics}}, Vol.~299, Birkh\"auser/Springer,
 Basel, 2012.

\bibitem{Mumford}
Mumford D., Geometric invariant theory, \textit{Ergebnisse der Mathematik und
 ihrer Grenzgebiete, Neue Folge}, Vol.~34, Springer-Verlag, Berlin~-- New
 York, 1965.

\bibitem{Neumann}
Neumann C., De problemate quodam mechanico, quod ad primam integralium
 ultraellipticorum classem revocatur, \href{http://dx.doi.org/10.1515/crll.1859.56.46}{\textit{J.~Reine Angew. Math.}}
 \textbf{56} (1859), 46--63.

\bibitem{Nijenhuis}
Nijenhuis A., {$X_{n-1}$}-forming sets of eigenvectors, \href{http://dx.doi.org/10.1016/S1385-7258(51)50028-8}{\textit{Indag. Math.}}
 \textbf{54} (1951), 200--212.

\bibitem{Olevskii}
Olevski\u\i M.N., Triorthogonal systems in spaces of constant curvature in which
 the equation {$\Delta_2u+\lambda u=0$} allows a complete separation of
 variables, \textit{Mat. Sb.} \textbf{27} (1950), 379--426.

\bibitem{Robertson}
Robertson H.P., Bemerkung \"uber separierbare {S}ysteme in der
 {W}ellenmechanik, \href{http://dx.doi.org/10.1007/BF01451624}{\textit{Math. Ann.}} \textbf{98} (1928), 749--752.

\bibitem{Schoebel12}
Sch{\"o}bel K., Algebraic integrability conditions for {K}illing tensors on
 constant sectional curvature manifolds, \href{http://dx.doi.org/10.1016/j.geomphys.2012.01.006}{\textit{J.~Geom. Phys.}} \textbf{62}
 (2012), 1013--1037, \href{http://arxiv.org/abs/1004.2872}{arXiv:1004.2872}.

\bibitem{Schoebel14}
Sch{\"o}bel K., The variety of integrable {K}illing tensors on the 3-sphere,
 \href{http://dx.doi.org/10.3842/SIGMA.2014.080}{\textit{SIGMA}} \textbf{10} (2014), 080, 48~pages, \href{http://arxiv.org/abs/1205.6227}{arXiv:1205.6227}.

\bibitem{Schoebel15}
Sch{\"o}bel K., An algebraic geometric approach to separation of variables,
 \href{http://dx.doi.org/10.1007/978-3-658-11408-4}{Springer Spektrum}, Wiesbaden, 2015.

\bibitem{Schoebel&Veselov}
Sch{\"o}bel K., Veselov A.P., Separation coordinates, moduli spaces and
 {S}tashef\/f polytopes, \href{http://dx.doi.org/10.1007/s00220-015-2332-x}{\textit{Comm. Math. Phys.}} \textbf{337} (2015),
 1255--1274, \href{http://arxiv.org/abs/1307.6132}{arXiv:1307.6132}.

\bibitem{Staeckel}
St{\"a}ckel P., Die {I}ntegration der {H}amilton--{J}acobischen
 {D}if\/ferentialgleichung mittelst {S}eparation der {V}ariablen,
 habilitationsschrift, Universit\"at Halle, 1891.

\bibitem{Stasheff}
Stashef\/f J.D., Homotopy associativity of {$H$}-spaces.~{I}, \href{http://dx.doi.org/10.1090/S0002-9947-1963-99939-9}{\textit{Trans.
 Amer. Math. Soc.}} \textbf{108} (1963), 275--292.

\bibitem{Stasheff+}
Stashef\/f J.D., Homotopy associativity of {$H$}-spaces.~{II}, \href{http://dx.doi.org/10.1090/S0002-9947-1963-0158400-5}{\textit{Trans.
 Amer. Math. Soc.}} \textbf{108} (1963), 293--312.

\bibitem{Vilenkin}
Vilenkin N.Ja., Special functions and group representation theory, Nauka,
 Moscow, 1965.

\bibitem{Vilenkin&Klimyk}
Vilenkin N.Ja., Klimyk A.U., Representation of {L}ie groups and special
 functions. {V}ol.~2. Class~I representations, special functions, and integral
 transforms, \href{http://dx.doi.org/10.1007/978-94-017-2883-6}{\textit{Mathematics and its Applications (Soviet Series)}},
 Vol.~74, Kluwer Academic Publishers Group, Dordrecht, 1993.

\bibitem{Wolf}
Wolf T., Structural equations for {K}illing tensors of arbitrary rank,
 \href{http://dx.doi.org/10.1016/S0010-4655(98)00123-4}{\textit{Comput. Phys. Comm.}} \textbf{115} (1998), 316--329.

\end{thebibliography}
\end{document}